\UseRawInputEncoding %arxiv
\documentclass[11pt,a4paper]{article}
\usepackage{amsfonts}
\usepackage{amssymb}
\usepackage{mathrsfs}
\usepackage{amsmath}
\usepackage{mathtools}
\usepackage{booktabs}
\usepackage{epsf,epsfig,amsfonts,amsgen,indentfirst}
\usepackage{amsmath,amstext,amsbsy,amsopn,amsthm,bbding,wasysym}
\usepackage{multicol,mathdots}
\usepackage{subfigure}
\usepackage{graphicx}
\allowdisplaybreaks

\newtheorem{theorem}{Theorem}[section]

\newtheorem{lemma}[theorem]{Lemma}

\newtheorem{remark}{Remark}

\begin{document}
\title
{\LARGE \textbf{Improved bounds for the coefficient of flow polynomials\thanks{ supported by NSFC (Nos. 12261071, 12471333)and NSF of Qinghai Province (No. 2020-ZJ-920).} }}

\author{ Tingzeng Wu$^{a,b}$\thanks{{Corresponding author.\newline
\emph{E-mail address}: mathtzwu@163.com, mashuang1502023@163.com, hjlai2015@notmail.com
}}, Shuang Ma$^a$, Hong-Jian Lai$^{c,d}$\\
{\small $^{a}$ School of Mathematics and Statistics, Qinghai Nationalities University, }\\
{\small  Xining, Qinghai 810007, P.R.~China} \\
{\small $^{b}$ Qinghai Institute of Applied Mathematics,   Xining, Qinghai, 810007, P.R.~China} \\
{\small $^{c}$ School of Mathematics and System Sciences, Guangdong Polytechnic Normal University, }\\
{\small Guangzhou 510665, China }\\
{\small $^{d}$ Department of Mathematics, West Virginia University, Morgantown, WV, USA }}
\date{}

\maketitle
\noindent {\bf Abstract:} Let $G$ be a connected bridgeless $(n,m)$-graph which may have loops and multiedges, and let $F(G,t)$ denote the flow polynomial of $G$. Dong and Koh \cite{Dong1} established an upper bound for the absolute value of   coefficient $c_{i}$ of $t^{i}$  in the expansion of $F(G,t)$,
where $0\leqslant i \leqslant m-n+1$. In this paper, we refine the aforementioned bound. Specifically, we demonstrate that when  $n \leqslant m \leqslant n+3$,  $|c_{i}|\leqslant d_{i}$, where  $d_{i}$ is the coefficient of $t^{i}$  in the expansion  $\prod\limits_{j=1}^{m-n+1}(t+j)$; and when  $m\geqslant n+4$,   $|c_{i}|\leqslant d_{i}$, with  $d_{i}$ being the coefficient of $t^{i}$  in the expansion $(t+1)(t+2)(t+3)^{2}(t+4)^{m-n-3}$. Furthermore, we prove that if $G$ is a connected bridgeless cubic graph having only real flow roots, then $b_{i}\leqslant |c_{i}|$, where  $b_{i}$ is the coefficient of $t^{i}$  in the expansion $(t+1)(t+2)^{\frac{n}{2}}$. Notably, if $G$ is simple connected bridgeless cubic graph  with only real flow roots, then $b_{i}$  is the coefficient of $t^{i}$  in the expansion  $(t+1)(t+2)^{\frac{n}{2}-2}(t+3)^{2}$.

%\smallskip
\section{Introduction}

For an arbitrary graph $G$, let $V(G)$, $E(G)$, $|V(G)|$ and $|E(G)|$ denote the vertex set, edge set, order and size of $G$ respectively. We call $G$ an $(n,m)$-graph if $|V(G)|=n$ and $|E(G)|=m$. Unless otherwise stated, the definitions and notations employed in this paper are in accordance with those in \cite{Dong1}.

Following Tutte \cite{Tutte2}, the flow polynomial of an $(n,m)$-graph $G$ is defined as
\begin{eqnarray}\label{ful1}
F(G,t)=\sum_{S\subseteq E(G)}(-1)^{m-|Z|}t^{|Z|+c(Z)-n},
\end{eqnarray}
where $c(Z)$ is the number of components in the spanning subgraph of $G$ induced by $Z$.
It follows from formula (\ref{ful1}) that $F(G,t)$ is a polynomial in $t$ of degree $m-n+c(G)$, where $c(G)$ denotes the number of components of $G$. If $G$ is a connected bridgeless $(n,m)$-graph, then
\begin{eqnarray}\label{equz1}
F(G,t)=\sum\limits_{i=0}^{m-n+1}(-1)^{m-n+1-i}a_{i}t^{i},
\end{eqnarray}
where $a_{m-n+1}=1$ and $a_{i}$'s are positive integers. As a research hotspot, flow polynomials of graphs have been studied extensively for many years. However, some problems about flow polynomials have not been solved. For these problems,
interested readers are referred to \cite{Dong2} and the references cited therein.

Dong and Koh \cite{Dong1} obtained an upper bound of the absolute value of the coefficient of  $t^{i}$ in the expansion $F(G,t)$.
\begin{theorem}(Dong and Koh,\cite{Dong1})\label{the1}
Let $G$ be any connected bridgeless $(n,m)$-graph. If
\begin{eqnarray*}
F(G,t)=\sum_{i=0}^{m-n+1}c_{i}t^{i},
\end{eqnarray*}
then $|c_{i}|\leqslant d_{i}$ for all $0\leqslant i\leqslant m-n+1$ holds, where
\begin{eqnarray}
\sum_{i=0}^{m-n+1}d_{i}t^{i}=
\begin{cases}
\prod\limits_{j=1}^{m-n+1}(t+j) & \text{if $n\leqslant m\leqslant n+1$};\\
(t+1)(t+2)(t+3)(t+4)^{m-n-2} & \text{otherwise}.\\
\end{cases}
\end{eqnarray}
\end{theorem}

In a similar vein to Theorem \ref{the1}, we have derived the following result.

\begin{theorem}\label{the2}
Let $G$ be any connected bridgeless $(n,m)$-graph. If
\begin{eqnarray}
F(G,t)=\sum_{i=0}^{m-n+1}c_{i}t^{i},
\end{eqnarray}
then $|c_{i}|\leqslant d_{i}$ for all $0\leqslant i\leqslant m-n+1$, where
\begin{eqnarray}
\sum_{i=0}^{m-n+1}d_{i}t^{i}=
\begin{cases}
\prod\limits_{j=1}^{m-n+1}(t+j) & \text{if $n\leqslant m\leqslant n+3$};\\
(t+1)(t+2)(t+3)^{2}(t+4)^{m-n-3} & \text{otherwise}.\\
\end{cases}
\end{eqnarray}
\end{theorem}

Let $G$ be a connected bridgeless cubic $(n,m)$-graph.  Dong and Koh \cite{Dong1}  determined an upper bound for the absolute value of the coefficient of  $t^{i}$ in the expansion $F(G,t)$  as follows.

\begin{theorem}(Dong and Koh,\cite{Dong1})\label{the3}
Let $G$ be any connected bridgeless cubic $(n,m)$-graph. If
\begin{eqnarray*}
F(G,t)=\sum_{i=0}^{m-n+1}c_{i}t^{i},
\end{eqnarray*}
then for all $0 \leqslant i \leqslant m-n+1$, we have $|c_{i}| \leqslant d_{i}$, where
\begin{eqnarray*}
\sum_{i=0}^{m-n+1}d_{i}t^{i}=
\begin{cases}
(t+1)(t+2) & if~|V(G)|=2;\\
(t+1)(t+2)(t+3)(t+4)^{(|V(G)|-4)/2} & otherwise.\\
\end{cases}
\end{eqnarray*}
\end{theorem}

It is a natural question to pose: for an arbitrary connected bridgeless cubic  $(n,m)$-graph $G$, how can we determine the lower bound of the absolute value of the coefficient of  $t^{i}$ in the expansion of  $F(G,t)$?  We shall use the term {\em flow root} of $G$ to refer to the zero of  $F (G,t)$.
%The study of flow roots of flow polynomials has attracted significant attention in the literature. For example, Wakelin \cite{Wake} provided insights into the distribution of the largest zero-free real flow roots of flow polynomials. Meanwhile, Dong \cite{Dong3} explored the distribution of flow roots within the real interval $(1, 2)$ for any bridgeless connected graph that has only real flow roots.
In this work, we focus on the problem of determining the lower bound of the absolute value of the coefficient of $t^{i}$  in the expansion of $F (G,t)$ as described above. We contribute to this area of research by presenting a result specifically for graphs that have only real flow roots.

\begin{theorem}\label{the4}
Let $\mathcal{G}$ denote the set consisting of  all connected bridgeless cubic $(n,m)$-graphs that possess only real flow roots. Suppose that $G \in \mathcal{G}$, and the flow polynomial of graph $G$ is
$F(G,t)=\sum\limits_{i=0}^{m-n+1}c_{i}t^{i}$.
For all $0 \leqslant i \leqslant m - n + 1$,  $b_{i} \leqslant |c_{i}|$, where $b_{i}$ is the the coefficient of  $t^{i}$ in the expansion
$
\sum\limits_{i=0}^{m-n+1}b_{i}t^{i}=(t+1)(t+2)^{n/2}.
$
\end{theorem}
\begin{remark}
Let $G_{1}$ and $G_{2}$ denote the two connected bridgeless cubic $(n,m)$-graphs, as depicted in Figure \ref{fig1}. The flow roots of $G_{1}$ and $G_{2}$  are \{1, $2^{n/2}$\}, where $a^{b}$ denote the fact that the multiplicity of the root $a$ is $b$. This means that the lower bound is reachable in Theorem \ref{the4}.
\end{remark}

\begin{figure}[htbp]
\begin{center}
\includegraphics[scale=0.23]{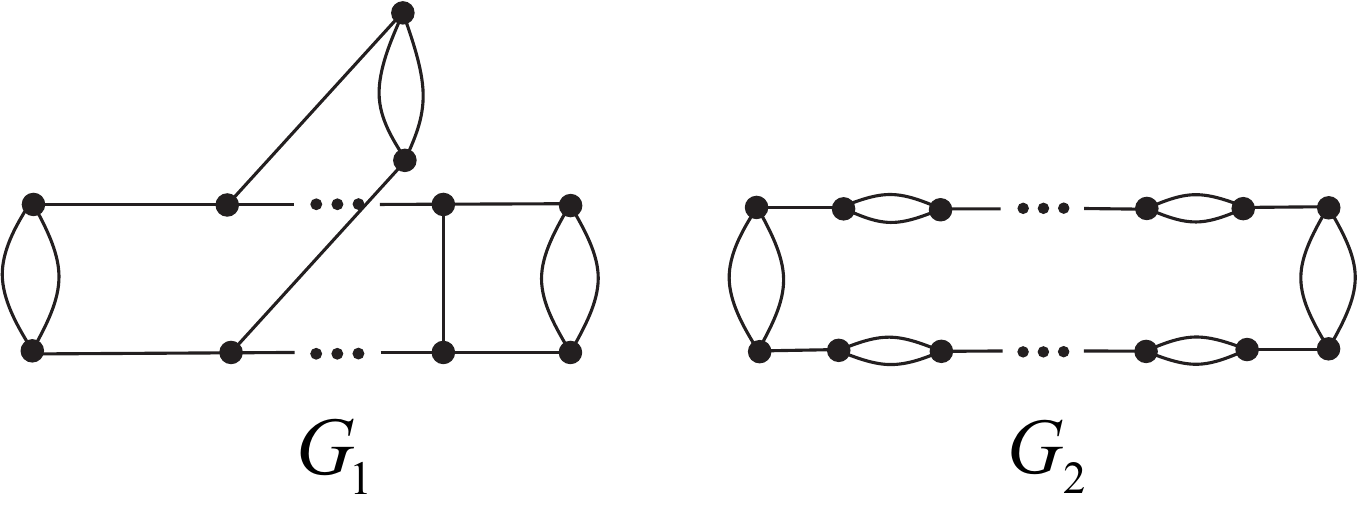}
\caption{\label{fig1}\small
{Connected bridgeless cubic $(n,m)$-graphs $G_{1}$ and $G_{2}$.}}
\end{center}
\end{figure}

Let $G^{*}$ denote the graph as depicted in Figure \ref{fig2}. When $G$ is simple connected bridgeless cubic $(n,m)$-graph  in Theorem \ref{the4}, the result presented in Theorem \ref{the4} can be improved in the following manner.

\begin{theorem}\label{the5}
Let $G$ be a connected simple bridgeless cubic $(n,m)$-graph that has only real flow roots. Denote the flow polynomial of graph $G$ as $F(G,t)=\sum_{i=0}^{m-n+1}c_{i}t^{i}$.
 For all $0 \leqslant i \leqslant m - n + 1$,  $b_{i} \leqslant |c_{i}|$,
where $b_{i}$ is the the coefficient of  $t^{i}$ in the expansion
$\sum\limits_{i=0}^{m-n+1}b_{i}t^{i}=(t+1)(t+2)^{n/2-2}(t+3)^{2}$ with  the equality holding  if and only if $G \cong G^{*}$.
\end{theorem}

\begin{figure}[htbp]
\begin{center}
\includegraphics[scale=0.15]{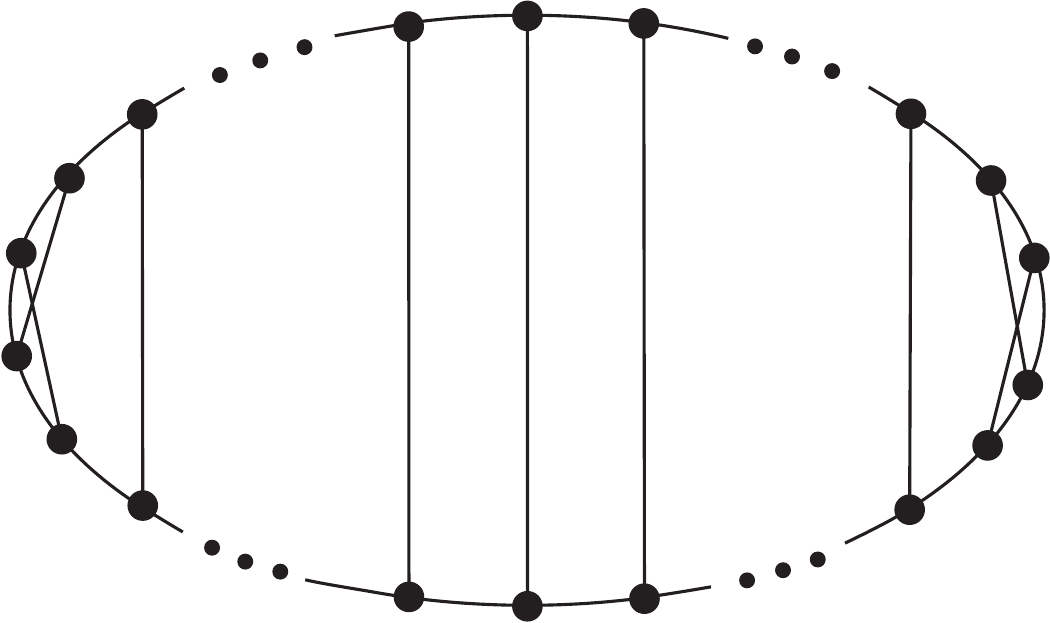}
\caption{\label{fig2}\small
{Graph $G^{*}$.}}
\end{center}
\end{figure}
The outline of this paper is shown as follows. In Section 2, we will present some properties and lemmas on the flow polynomial. We give the proof of Theorem \ref{the2} in section 3. We investigate some properties of the flow polynomial of connected bridgeless cubic graphs, and the proofs of Theorems \ref{the4} and \ref{the5} are given in section 4. In the final section, we give a brief summary.

\section{Preliminaries}

In this section, we present related definitions, and we recall and develop some results, which are useful in the proofs of the main results later.

Let $G$ be a connected graph. A {\em vertex cut} of  graph $G$ is a set $S\subseteq V(G)$ such that $G-S$ has more than one component. A {\em block} of graph $G$ is a maximal connected subgraph of $G$ that has no vertex cut. An {\em edge cut} of  graph $G$ is a set $S'\subseteq E(G)$ such that $G-S'$ has more than one component. Further, if $G - S'$ has isolated vertices, then we call the edge cut set $S'$ of graph $G$ a {\em proper edge cut set}.
Let $u$ be a vertex in $G$ such that $e_{1}=uw_{1}$ and $e_{2}=uw_{2}$ are the only edges incident with $u$ and both are not loops (it is possible that $w_{1}=w_{2}$). A \textit{desubdivision} at $u$ is an operation of $G$ that deletes $u$ and adds one edge joining $w_{1}$ to $w_{2}$. So, if $G'$ is the graph obtained from $G$ by a \textit{desubdivision} at $u$, then $G$ is a subdivision of $G'$.
Let $G$ be a graph and $x \in V(G)$. If $d(u) = 3$ for all $u \in V(G)-\{x\}$, then $G$ is called a \textit{near-cubic graph} at $x$. If $d(x) = 3$, then the \textit{near-cubic graph} $G$ at $x$ is a cubic graph. 
For any $w \in V(G)$, let $N(w)$ be the set of all neighbors of $w$ excluding $w$.
Let $Z_{3}$ is a graph obtained from two vertices by adding three edges to the two vertices.

Tutte \cite{Tutte2} provided the following properties of the flow polynomial $F(G,t)$:
\begin{eqnarray}\label{ful5}
F(G,t)=
\begin{cases}
1,& if ~G~ is ~empty;\\
0,& if ~G~ has~ a~ bridge;\\
F(G_{1},t)\cdots F(G_{k},t),& if ~G=G_{1}\cup\cdots\cup G_{k};\\
(t-1)F(G-e,t),& if ~e~ is ~a~ loop;\\
F(G/e,t)-F(G-e,t),& otherwise.\\
\end{cases}
\end{eqnarray}
where $G/e$ and $G-e$  are the graphs obtained from $G$ by contracting $e$ and deleting $e$, respectively, and $G_{1} \cup G_{2} \cup \cdots \cup G_{k}$ is the disjoint union of graphs $G_{1}, G_{2}, \cdots, G_{k}$.

Let $\mathcal{P}$ be the set of all polynomials in $t$. Define a binary relation $\leqslant_{c}$ in $\mathcal{P}$ as follows,
\begin{eqnarray*}
\sum^{r}_{i=0}a_{i}t^{i}\leqslant_{c}\sum^{r}_{i=0}a'_{i}t^{i}
\end{eqnarray*}
if and only if $a_{i}\leqslant a'_{i}$ for all $i = 0, 1, 2, \cdots, r$. Clearly, $(\mathcal{P}, \leqslant_{c})$ is a partially ordered set.

For a connected $(n,m)$-graph $G$, write
\begin{eqnarray}\label{fu3}
\tau(G,t)=(-1)^{m-n+1}F(G,-t).
\end{eqnarray}
\noindent If $G$ is bridgeless, by (\ref{equz1}),  $\tau(G,t)$ is a polynomial in  $t$ of degree $m-n+1$ in which all  coefficients are positive integers.

From formula (\ref{fu3}), Theorem \ref{the2} can be rewritten as

\begin{lemma}\label{lem1}
Theorem \ref{the2} holds if
\begin{eqnarray*}
\tau(G,t)\leqslant_{c}
\begin{cases}
\prod\limits_{j=1}^{m-n+1}(t+j) & if ~m-n\leqslant 3,\\
(t+1)(t+2)(t+3)^{2}(t+4)^{m-n-3} & otherwise.\\
\end{cases}
\end{eqnarray*}
holds for each connected bridgeless $(n,m)$-graph $G$.
\end{lemma}

\begin{lemma}\label{lem2}(Dong and Koh, \cite{Dong1})
If $G$ is a connected graph and $e$ is an edge in $G$ that is neither a loop nor a bridge, then
$
\tau(G,t)=\tau(G/e,t)+\tau(G-e,t).
$
In particular, if $G - e$ has a bridge, then $\tau(G, t) = \tau(G/e, t)$.
\end{lemma}

\begin{lemma}\label{lem3}(Dong and Koh, \cite{Dong1})
If $G'$ is the graph obtained from a graph $G$ by a \textit{desubdivision} at a vertex in $G$, then
$
F(G,t)=F(G',t), \tau(G,t)=\tau(G',t).
$
\end{lemma}

\begin{lemma}\label{lem4}(Dong and Koh, \cite{Dong1})
Let $G$ be a connected bridgeless $(n, m)$-graph. If $m > n$, then there exists a connected cubic graph $H$ such that $|E(H)| - |V(H)| = m - n$, and $\tau(G, t) \leqslant_{c} \tau(H, t)$.
\end{lemma}

\begin{lemma}\label{lem5}
Let $G$ be a 2-edge-connected bridgeless cubic $(n,m)$-graph. If $m>n$, then there exists a 3-edge-connected cubic graph $H$ such that  $|c_{i}|\leqslant|c^{'}_{i}|$, where $c_{i}$ and $c^{'}_{i}$  denote respectively the $i$-th coefficient of $F(G,t)$ and $F(H,t)$, where $i=1,2, \cdots, m-n+1$.
\end{lemma}

\begin{proof}
Suppose that $e$ is an edge of some 2-edge cut in $G$. Checking the structures of $G/e$ and $G-e$, we know that $G/e$ is a connected bridgeless near-cubic  graph, and $G-e$ is a bridge graph. By (\ref{ful5}), we  get that $F(G,t) = F(G/e,t)-F(G-e,t)=F(G/e,t)$. By Lemma \ref{lem4} and (\ref{fu3}), we obtain that there exists a 3-edge-connected cubic graph $H$ such that  $|c_{i}|\leqslant|c^{'}_{i}|$, where $c_{i}$ and $c^{'}_{i}$  denote respectively the $i$-th coefficient of $F(G,t)$ and $F(H,t)$, where $i=1,2, \cdots, m-n+1$.
\end{proof}

%Dong and Koh \cite{Dong1} exploited the above-mentioned property to characterize the following important result,
Let $G$ be a graph and $x\in V(G)$. If $d(u)=3$ for all $u\in V(G)\setminus\{x\}$, then $G$ is call a {\em near cubic graph} at $x$.
\begin{theorem}\label{lem6}(Dong and Koh, \cite{Dong1})
Let $G$ be any connected near-cubic $(n, m)$-graph at  $x$. Then
\begin{eqnarray*}
\tau(G,t)\leqslant_{c}(t+2)(t+1)^{\lfloor(d(x)-1)/2\rfloor}(t+4)^{\lfloor((n-1)/2\rfloor}.
\end{eqnarray*}
\end{theorem}

\begin{lemma}\label{lem8}(Dong and Koh, \cite{Dong1})
Let $G$ be a connected \textit{near-cubic graph} at $x$. If $|V(G)| \geqslant 3$ and $G$ has no bridges, then $N(x) \geqslant 2$.
\end{lemma}

\begin{lemma}\label{lem9}(Dong and Koh, \cite{Dong1})
Let $G$ be a connected \textit{near-cubic graph} at $x$. Then $d(x)$ and $|V(G)|$ have different parities, and
\begin{eqnarray*}
|E(G)|-|V(G)|+1=(d(x)+|V(G)|-1)/2.
\end{eqnarray*}
\end{lemma}

%\begin{lemma}\label{lem10}(Dong, \cite{Dong3})
%Let $G$ be a bridgeless graph with $n$ vertices and $m$ edges, and all the flow roots of $G$ are real numbers. Then all the flow roots of $G$ are integers if and only if $G \in \{L, Z_{3}, K_{4}\}$, where $L$ is a loop, and $Z_{3}$ is a graph obtained from two vertices by joining three edges two vertices.
%\end{lemma}

\begin{lemma}\label{lem11}(Sekine and Zhang, \cite{Zhang})
Let $G$ be a graph with edge cut $A$ with $2 \leqslant |A| \leqslant 3$. Let $G_{1}$ and $G_{2}$ be formed from $G$ by respectively contracting one component of $G - A$. Then
\begin{eqnarray}
F(G,t)=\frac{F(G_{1},t)F(G_{2},t)}{F(K_{2}^{h},t)}.
\end{eqnarray}
where $h=|A|$, and $F(K_{2}^{2},t)=(t-1), F(K_{2}^{3},t)=(t-1)(t-2)$.
\end{lemma}

\begin{figure}[htbp]
\begin{center}
\includegraphics[scale=0.3]{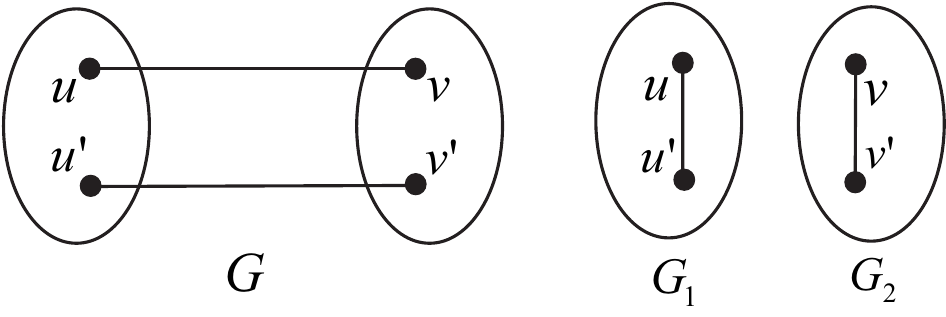}
\caption{\label{fig3}\small
{\textit{2-edge-cut decomposition} operation of $G$ }}
\end{center}
\end{figure}

Let $G$ be a 2-edge-connected cubic graph. Define an operation for  graph $G$ as follows. Suppose that $S=\{uv, u^{'}v^{'}\}$ be a 2-edge cut of $G$. Graphs $G_{1}$ and $G_{2}$ are two graphs obtained from $G$ by deleting edges $uv$ and $u^{'}v^{'}$, and adding edges $uu^{'}$ and $vv^{'}$, The resulting graphs are illustrated in Figure \ref{fig3}.
Then, we say that $G$ is decomposed to graphs $G_{1}$ and $G_{2}$ by a {\em 2-edge-cut decomposition}. In particular, repeating  apply  this  operation to every  2-edge cut of $G$, $G$ is decomposed into 3-edge-connected cubic graphs $G_{1}$, $G_{2}$ and $G_{k+1}$, where $k$ is the number of 2-edge cuts in $G$. Repeated application of the Lemmas \ref{lem3} and Lemma \ref{lem11}, we can obtain a result as follows.
\begin{lemma}\label{lem12}
Let $G$ be a 2-edge-connected cubic graph with $k$ 2-edge cuts. Denote by $G_{1}$, $G_{2}, \cdots, G_{k+1}$  all 3-edge-connected cubic graphs that are derived from   the 2-edge-cut decomposition of  $G$. Then
\begin{eqnarray}\label{zhang2}
F(G,t)=\frac{\prod\limits_{i=1}^{k+1}F(G_{i},t)}{(t-1)^{k}}.
\end{eqnarray}
\end{lemma}

\begin{lemma}(Wakelin, \cite{Wake})\label{lem13}
Let $G$ be a connected bridgeless $(n,m)$-graph with $b$ blocks. Then\\
(a) $(-1)^{m-n+1}F(G,t) > 0$ holds for all real numbers $t$ in the interval $(-\infty, 1)$. \\
(b) $F(G,t)$ has a zero of multiplicity $b$ at $t=1$. \\
(c) $(-1)^{m-n+b+1}F(G,t) > 0$ holds for all real numbers $t$ in (1, 32/27].
\end{lemma}

%Jackson \cite{Jackson1} and Jackson \cite{Jackson2} provided the distribution of the flow roots  for connected bridgeless \textit{near-cubic graphs} (cubic graphs, respectively) and presented the following significant results.

\begin{lemma}(Jackson, \cite{Jackson1})\label{lem14}
If $G$ is a bridgeless graph with at most one vertex of degree larger than 3, then G has no flow roots in the interval $(1, 2)$.
\end{lemma}

\begin{lemma}(Jackson,\cite{Jackson1})\label{lem15}
The flow root of a graph $G$ is 2 if and only if $G$ is 3-connected.
\end{lemma}

\begin{lemma}\label{lem16}(Jackson, \cite{Jackson2})
Let $G$ be a 3-connected cubic graph with $n$ vertices and $m$ edges. Then,\\
(a). $F(G,t)$ is non-zero with sign $(-1)^{m-n}$ for $t\in (1,2)$. \\
(b). $F(G,t)$ has a zero of multiplicity 1 at $t = 2$; \\
(c). $F(G,t)$ has no flow roots in the interval $t \in (2,\delta)$, where $\delta \approx 2.546$ is the flow root of the cubic graph in the interval $(2, 3)$.
\end{lemma}

\begin{lemma}\label{lem17}(Fuchs and Schwarz, \cite{Fuchs})(Vieta theorem)

Let ~$f(x)=a_{n}x^{n}+a_{n-1}x^{n-1}+\cdots +a_{1}x+a_{0}$ be a polynomial of degree $n$ with roots $\lambda_{1},\lambda_{2},\cdots,\lambda_{n}$, where $a_{n},\cdots,a_{0}$ are real coefficients. Then $\sum\limits_{i_{1}<i_{2}<\cdots<i_{j}}\lambda_{i_{1}}\lambda_{i_{2}}\cdots\lambda_{i_{j}}=(-1)^{n-j}\frac{a_{n-j}}{a_{n}}$.
where $j=1,2,\cdots,n$.
\end{lemma}

\section{The proof of Theorem \ref{the2}}

Before the proof of Theorem \ref{the2}, we introduce two results.
Theorem \ref{lem6} will be improved as follows.

\begin{theorem}\label{the6}
Let $G$ be any connected \textit{near-cubic $(n, m)$-graph} at $x$. Then
\begin{eqnarray}\label{fum}
\tau(G,t)\leqslant_{c}(t+2)(t+3)(t+1)^{\lfloor(d(x)-1)/2\rfloor}(t+4)^{\lfloor(n-3)/2\rfloor}.
\end{eqnarray}
\end{theorem}

\begin{proof}
By (\ref{ful5}), the result hold, if $G$ has a bridge. Thus, we assume that $G$ contains no bridges and, in particular, $d(x)\geqslant 2$.

We shall prove this result by contradiction. Suppose that the result is false and $G$ is a counterexample with $|V(G)|$ being minimized.

\textbf{Claim 1}. $|V(G)|\geqslant 5$

Suppose $n = 1$. By Lemma \ref{lem9}, $d(x)$  is an even integer at least $2$ and

\begin{eqnarray*}
F(G,t)=(q-1)^{d(x)/2}.
\end{eqnarray*}

By Lemma \ref{lem9}, we have
\begin{eqnarray*}
\tau(G,t)=(-1)^{(d(x)+1-1)/2}F(G,-t)
         =(t+1)^{d(x)/2}
         \leqslant_{c}(t+1)^{\lfloor (d(x)-1)/2\rfloor}(t+2).
\end{eqnarray*}
Thus (\ref{fum}) holds for $G$ if $n = 1$, a contradiction.

Suppose $n=2$. Then, by Lemma \ref{lem9}, $d(x)$ is an odd integer at least $3$. Furthermore,

\begin{eqnarray*}
F(G,t)=(t-1)^{(d(x)-1)/2}(t-2).
\end{eqnarray*}
By Lemma \ref{lem9}, we have
\begin{eqnarray*}
\tau(G,t)=(-1)^{(d(x)+2-1)/2}F(G,-t)\leqslant_{c}(t+1)^{(d(x)-1)/2}(t+2).
\end{eqnarray*}
Thus (\ref{fum}) holds for $G$ if $n = 2$, a contradiction.

Suppose $n = 3$. By Lemma \ref{lem9}, $d(x) \geqslant 2$ is an even integer, and

\begin{eqnarray*}
F(G,t)=(t-1)^{d(x)/2}(t-2).
\end{eqnarray*}
By Lemma \ref{lem9}, we have
\begin{eqnarray*}
\tau(G,t)=(-1)^{(d(x)+3-1)/2}F(G,-t)
         =(t+1)^{d(x)/2}(t+2)
         \leqslant_{c}(t+1)^{\lfloor (d(x)-1)/2\rfloor}(t+2)(t+3).
\end{eqnarray*}
Thus (\ref{fum}) holds for $G$ if $n = 3$, a contradiction.

Suppose $n=4$. Then $d(x) \geqslant 3$ is an odd integer by Lemma \ref{lem9}. There exists two non-isomorphic cubic graphs with 4 vertices. One, denoted by $L^{'}_{4}$, is obtained by adding a multiple  edge to each pair of opposite edges on the cycle $C_{4}$ and adding $k$ ($k \geqslant 0$) loops to some vertex $u$ of $C_{4}$. The other, denoted by  $K'_{4}$, is obtained by adding $k$ ($k \geqslant 0$) loops to  vertex $v$ of $K_{4}$. Direct computing yields $
F(L'_{4},t)=(t-1)^{k+1}(t-2)^{2}=(t-1)^{(d(u)-1)/2}(t-2)^{2}$ and $
F(K'_{4},t)=(t-1)^{k'+1}(t-2)(t-3)=(t-1)^{(d(v)-1)/2}(t-2)(t-3)$.

By Lemma \ref{lem9}, we have
$
\tau(G,t)=(-1)^{(d(x)+4-1)/2}F(G,-t)\leqslant_{c}(t+1)^{(d(x)-1)/2}(t+2)(t+3)$.
Thus  (\ref{fum}) holds for $G$ if $n = 4$, which is a contradiction.
Claim 1 thus follows.

\textbf{Claim 2}. $d(x)\geqslant 3$.

Suppose that $d(x) = 2$. Let $H$ be the graph obtained from $G$ by a \textit{desubdivision} at $x$. Then $H$ is a connected cubic graph with $|V(H)| = |V(G)| - 1$ ($|V(G)| \geqslant 5$). This implies that $|V(G)|$ is odd and the result holds for $H$. By Lemma \ref{lem3},

\begin{eqnarray*}
\tau(G,t)=\tau(H,t)
         &\leqslant_{c}&(t+1)(t+2)(t+3)(t+4)^{\lfloor(n-4)/2\rfloor}\\
         &\leqslant_{c}&(t+1)^{\lfloor (d(x)-1)/2\rfloor}(t+2)(t+3)(t+4)^{\lfloor(n-3)/2\rfloor},
\end{eqnarray*}
i.e., the result holds for $G$, a contradiction. Claim 2 thus follows.

By Claim 1, $|V(G)|\geqslant 5$. By Lemma \ref{lem8}, there exist two distinct vertices $u$ and $v$ in $N(x)$. Let $e_{1}=xu$ and $e_{2}=xv$ be two edges as shown in Figure \ref{fig4}.
\begin{figure}[htbp]
\begin{center}
\includegraphics[scale=0.25]{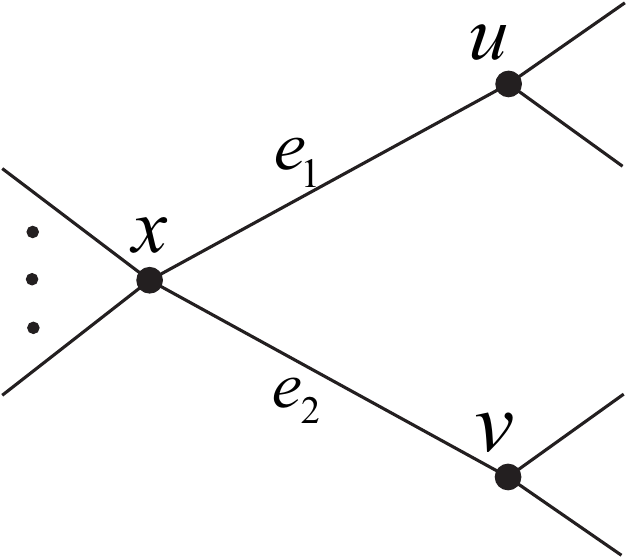}
\caption{\label{fig4}
\small{Partial graph of  \textit{near-cubic} graph at $x$.}}
\end{center}
\end{figure}

Observe that $G/e_{1}$ is a connected near-cubic graph at $x$ with $d_{G/e_{1}}(x) = d(x) + 1$ and $|V(G/e_{1})|=|V(G)|-1$. Thus the result holds for $G/e_{1}$, i.e.,
\begin{eqnarray*}
\tau(G/e_{1},t)\leqslant_{c}(t+2)(t+3)(t+1)^{\lfloor(d(x)-1)/2\rfloor}(t+4)^{\lfloor(n-3)/2\rfloor}.
\end{eqnarray*}

If $G - e_{1}$ has a bridge, then, by Lemma \ref{lem2},
\begin{eqnarray*}
\tau(G,t)&=&\tau(G/e_{1},t)\\
         &\leqslant_{c}&(t+2)(t+3)(t+1)^{\lfloor(d_{G/e}(x)-1)/2\rfloor}(t+4)^{\lfloor(|V(G/e)|-3)/2\rfloor}\\
         &=&(t+2)(t+3)(t+1)^{\lfloor(d(x))/2\rfloor}(t+4)^{\lfloor(n-4)/2\rfloor}\\
         &\leqslant_{c}&(t+2)(t+3)(t+1)^{\lfloor(d(x)-1)/2\rfloor}(t+4)^{\lfloor(n-3)/2\rfloor}.
\end{eqnarray*}
Thus the result holds for $G$,  a contradiction. This implies that $G-e_{1}$ has no bridges.

By Lemma \ref{lem2},
\begin{eqnarray}\label{fum1}
\tau(G,t)&=&\tau(G/e_{1},t)+\tau(G-e_{1},t)\\\nonumber
         &=&\tau(G/e_{1}/e_{2},t)+\tau(G/e_{1}-e_{2},t)+\tau(G/e_{2}-e_{2},t)+\tau(G-e_{1}-e_{2},t).
\end{eqnarray}
Let $G_{0}=G/e_{1}/e_{2}$, where $G_{0}$ is a connected near-cubic graph with $|V(G_{0})|=n-2$ and $d_{G_{0}}(x)=d(x)+2$. In the graph $G/e_{1}-e_{2}$, $v$ is of degree 2 and every vertex in $V(G)- \{x,v\}$ is of degree 3. Let $G_{1}$ be the graph obtain from $G/e_{1}-e_{2}$ by the desubdivision at vertex $v$. Then $G_{1}$ is a connected near-cubic graph at $x$ with $|V(G_{1})|=n-2$ and $d_{G_{1}}(x)=d(x)$. Similarly, let $G_{2}$ be the graph obtained from $G/e_{2}-e_{1}$ by a desubdivision at vertex $u$ and observe that $G_{2}$ is a connected near-cubic graph at $x$ with $|V(G_{2})|=n-2$ and $d_{G_{2}}(x)=d(x)$. In the graph $G-e_{1}-e_{2}$, $u$ and $v$ are of degree 2 and every vertex in $V(G)-\{x,u,v\}$ is of degree 3. Let $G_{3}$ be the graph obtained from $G-e_{1}-e_{2}$ by the desubdivision at vertices $u$ and $v$ successively. Then $G_{3}$ is a connected near-cubic graph at $x$ with $|V(G_{3})|=n-2$ and $d_{G_{3}}(x)=d(x)-2>0$.

Since $|V(G_{i})| = n - 2$, the result holds for all $G_{i}$, $i = 0, 1, 2, 3$, i.e.,
\begin{eqnarray}\label{fum2}
\tau(G_{i},t)\leqslant_{c}(t+2)(t+3)(t+1)^{\lfloor(d_{G_{i}}(x)-1)/2\rfloor}(t+4)^{\lfloor(|V(G_{i})|-3)/2\rfloor}.
\end{eqnarray}

By (\ref{fum1}) and (\ref{fum2}), we have
\begin{eqnarray*}
\tau(G,t)&\leqslant_{c}&(t+2)(t+3)(t+1)^{\lfloor(d(x)+1)/2\rfloor}(t+4)^{\lfloor(n-5)/2\rfloor}\\
                 &+&2(t+2)(t+3)(t+1)^{\lfloor(d(x)-1)/2\rfloor}(t+4)^{\lfloor(n-5)/2\rfloor}\\
                 &+&(t+2)(t+3)(t+1)^{\lfloor(d(x)-3)/2\rfloor}(t+4)^{\lfloor(n-5)/2\rfloor}\\
         &=&(t+2)^{3}(t+3)(t+1)^{\lfloor(d(x)-3)/2\rfloor}(t+4)^{\lfloor(n-5)/2\rfloor}\\
         &\leqslant_{c}&(t+2)(t+3)(t+1)^{\lfloor(d(x)-1)/2\rfloor}(t+4)^{\lfloor(n-3)/2\rfloor}.
\end{eqnarray*}
Thus the result holds for $G$, a contradiction. The proof is thus complete.
\end{proof}

\begin{theorem}\label{th7}
Let $G$ be a connected bridgeless $(n,m)$-cubic graph. Then $\tau(G,t)$ is of degree $n/2 + 1$, and
\begin{eqnarray*}
\tau(G,t)\leqslant_{c}
\begin{cases}
(t+1)(t+2) & if~n=2,\\
(t+1)(t+2)(t+3) & if~n=4,\\
(t+1)(t+2)(t+3)(t+4) & if~n=6,\\
(t+1)(t+2)(t+3)^{2}(t+4)^{(n-6)/2} & otherwise.\\
\end{cases}
\end{eqnarray*}
\end{theorem}

\begin{proof}
If $G$ has a loop at some vertex $u$, then $G$ has a bridge incident with $u$, a contradition. Hence $G$ has no loops.

Since $G$ is a cubic graph, $|V(G)|$ is even. By Lemma \ref{lem9}, the degree of $\tau(G,t)$ is
~$(3+n-1)/2=n/2+1$.

If $n = 2$, then $G\cong Z_{3}$. In this case, $F(G, t) = (t-1)(t-2)$ and $\tau(G,t)=(t+1)(t+2)$.

If $n = 4$, then there exists two non-isomorphic connected cubic graphs. One, denoted by $L_{4}$, is  obtained by adding a parallel edge to each pair of opposite edges on the cycle $C_{4}$. The other is $K_{4}$. Direct computing yields that
$F(L_{4},t)=(t-1)(t-2)^{2}$ and
$F(K_{4},t)=(t-1)(t-2)(t-3)$.
Hence $
\tau(G,t)\leqslant_{c}(t+1)(t+2)(t+3)$.

\begin{figure}[htbp]
\begin{center}
\includegraphics[scale=0.3]{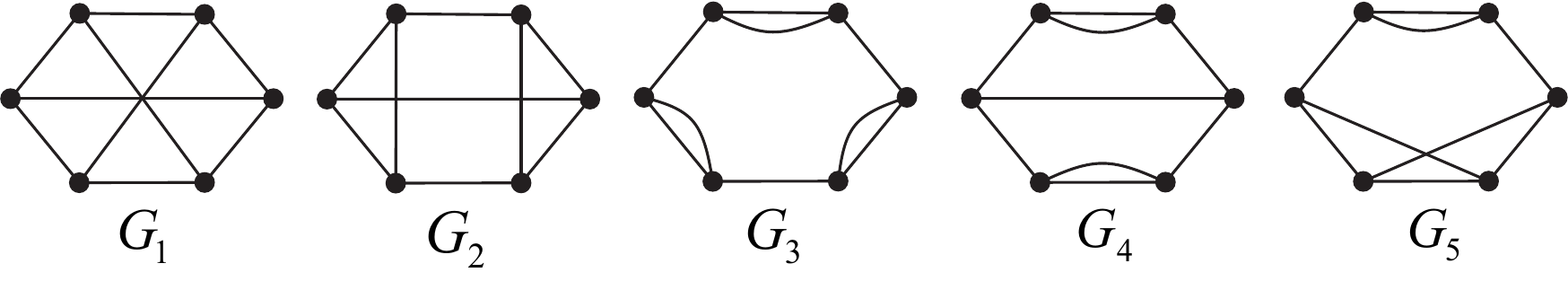}
\caption{\label{fig5}\small
{All cubic graphs with $n = 6$.}}
\end{center}
\end{figure}

If $n = 6$, then there exists five non-isomorphic  cubic graphs with 6 vertices, as depicted in Figure \ref{fig5}. Direct computing yields that $F(G_{1}, t) = (t - 1)(t - 2)(t^2 - 6t + 10)$,  $F(G_{2}, t) = (t - 1)(t - 2)(t - 3)^2$, $F(G_{3}, t) = (t - 1)(t - 2)^3$, $F(G_{4}, t) = (t - 1)(t - 2)^3$, and  $F(G_{5}, t) = (t - 1)(t - 2)^2(t - 3)$. Then,
$
\tau(G,t)\leqslant_{c}(t+1)(t+2)(t^{2}+6t+10)\leqslant_{c}(t+1)(t+2)(t+3)(t+4)$.

Now we assume that $n \geqslant 8$. Since there are no three edges joining the same two vertices. Thus each vertices in $G$ is adjacent to at least two distinct vertices. Choose any $x \in V(G)$. Let $u, v \in N(x)$, and $u \neq v, e_{1} = xu, e_{2} = xv$, as shown in Figure \ref{fig4}. Since the graph $G$ is a cubic graph, $G - e_{1} - e_{2}$ has a bridge. By Lemma \ref{lem2},
\begin{eqnarray}\label{ful21}
\tau(G,t)=
\begin{cases}
\tau(G/e_{1}/e_{2},t)+\tau(G/e_{1}-e_{2},t) & if ~G-e_{1}~ has ~a~ bridge,\\
\tau(G/e_{1}/e_{2},t)+\tau(G/e_{1}-e_{2},t)+\tau(G/e_{2}-e_{1},t) & otherwise.\\
\end{cases}
\end{eqnarray}

Let $G_{0}=G/e_{1}/e_{2}$. $G_{0}$ is a connected near-cubic graph at vertex $x$ with $|V(G_{0})|=n-2$ and $d_{G_{0}}(x)=5$. By Theorem \ref{the6},
\begin{eqnarray}\label{ful22}
\tau(G_{0},t)\leqslant_{c}(t+2)(t+3)(t+1)^{2}(t+4)^{\lfloor(n-5)/2\rfloor}=(t+1)^{2}(t+2)(t+3)(t+4)^{(n-6)/2}.
\end{eqnarray}

Observe that $G/e_{1}-e_{2}$ be a connected near-cubic graph at $v$ , where $d(v) = 2$. Let $G_{1}$ be the graph obtained from $G/e_{1}-e_{2}$ by a desubdivision at $v$. Similarly, $G/e_{2}-e_{1}$ is a connected near-cubic graph at $u$, where $d(u) = 2$. Let $G_{2}$ be the graph obtained from $G/e_{2}-e_{1}$ by the desubdivision at $u$. Note that, for $i = 1, 2$, $G_{i}$ is a cubic graph with $|V(G_{i})| = n - 2$, and by Theorem \ref{the6},
\begin{eqnarray}\label{ful23}
\tau(G_{i},t)\leqslant_{c}(t+2)(t+3)(t+1)(t+4)^{\lfloor(n-5)/2\rfloor}=(t+1)(t+2)(t+3)(t+4)^{(n-6)/2}.
\end{eqnarray}

By formulas (\ref{ful21}), (\ref{ful22}) and (\ref{ful23}),
\begin{eqnarray*}\label{ful2}
\tau(G,t)&\leqslant_{c}&(t+2)(t+3)(t+1)^{2}(t+4)^{(n-6)/2}+2(t+1)(t+2)(t+3)(t+4)^{(n-6)/2}\\
         &=&(t+1)(t+2)(t+3)^{2}(t+4)^{(n-6)/2}.
\end{eqnarray*}
The proof is thus complete.
\end{proof}

\textbf{Proof of Theorem \ref{the2}}.
It is routine to verify that if $m = n$, then $F(G, t) = t - 1$; and that if $m = n + 1$, then either $F(G, t) = (t - 1)^2$ or $F(G, t) = (t - 1)(t - 2)$. Thus the result holds if $m-n\leqslant 1$.

Assume that $m = n + 2$. By Lemma \ref{lem4},  for any connected bridgeless $(n, m)$-graph, there exists a connected cubic graph $H$ such that $|E(H)| - |V(H)| = 2$ and the absolute value of each coefficient in the expansion  $F(H, t)$ is  greater than  the absolute value of each coefficient in the expansion  $F(G, t)$. By the degree-sum formula, we obtain that $|V(H)|=4$.  By Theorem \ref{th7},
we know that the absolute value of  coefficient $c_{i}$ of $t^{i}$ in the expansion of $F(G, t)$ is no greater than the value of the coefficient $d_{i}$ of $t^{i}$ in the expansion of $(t + 1)(t +2)(t+3)$, where $0\leqslant i \leqslant m-n+1$. Thus the result holds if $m-n=2$.

Suppose that $m = n + 3$. Similar to the proof as above, by Lemma \ref{lem4},  for any connected bridgeless $(n, m)$-graph,  we know that there exists a connected cubic graph $H$ such that $|E(H)| - |V(H)| = 3$ and the absolute value of each coefficient in the expansion  $F(H, t)$ is  greater than  the absolute value of each coefficient in the expansion  $F(G, t)$. By the degree-sum formula, we obtain that $V(H)=6$.  By Theorem \ref{th7},
we know that the absolute value of  coefficient $c_{i}$ of $t^{i}$ in the expansion of $F(G, t)$ is no greater than the value of the coefficient $d_{i}$ of $t^{i}$ in the expansion of $(t + 1)(t +2)(t+3)(t+4)$, where $0\leqslant i \leqslant m-n+1$. Thus the result holds if $m-n=3$.

Now let $m-n \geqslant 4$. If $G$ is a cubic graph, then the result holds by Theorem \ref{th7}. If $G$ is not cubic graph, then the result follows from Lemma \ref{lem4} and Theorem \ref{th7}.
\hfill $\square$

\section{The proofs of Theorems \ref{the4} and \ref{the5}}\label{sec4}

In this section,  we will give the proofs of Theorems \ref{the4} and \ref{the5}.

{\bf Proof of Theorem \ref{the4}}. Let $G\in \mathcal{G}$ be a connected bridgeless cubic graph having only real roots. By applying the 2-edge-cut decomposition to $G$, we find that $G$ can be decomposed into 3-edge-connected bridgeless cubic graphs $G_{1}, G_{2}, \cdots, G_{k+1}$ in which $k$ represents  the number of 2-edge cuts in $G$.
Checking the structure of $G_{i}$, if $|V(G_{i})|=2$ then $G_{i}\cong Z_{3}$. Directly computing yields $F(Z_{3},t)=(t-1)(t-2)$. By Lemmas \ref{lem12}, \ref{lem13}, \ref{lem14}, \ref{lem15} and \ref{lem16}, we conclude that $F(G,t)=0$ has precisely one root equal to 1 and  $n/2$ roots, each of which is greater than or equal to 2. This implies that the absolute value of   coefficient $c_{i}$ of $t^{i}$  in the expansion of $F(G,t)$ greater than or equal to the coefficient   $b_{i}$ of $t^{i}$  in the expansion of $(t+1)(t+2)^{n/2}$ for each $i$ with $0 \leqslant i\leqslant m-n+1$.\hfill $\square$

\begin{lemma}\label{art4}
Let $\mathscr{G}_{n,m}$ denote the collection of 2-edge-connected simple bridgeless $(n,m)$-cubic graphs. If $G \in \mathscr{G}_{n,m}$, then $G$ contains at most $\frac{n}{2}-3$ 2-edge cuts if and only if $G \cong G^{*}$.
\end{lemma}

\begin{proof}

Let  $G \in \mathscr{G}_{n,m}$ be a 2-edge-connected simple bridgeless cubic $(n,m)$-graph with $k$ 2-edge cuts. By applying the 2-edge-cut decomposition to $G$, we find that $G$ can be decomposed into $k + 1$ 3-edge-connected bridgeless cubic graphs, denoted as $G_{1}$, $\cdots$, $G_{k}$ and $G_{k + 1}$. Since $G_{i}$ is a connected cubic graph, we know that $|V(G_{i})| \geqslant 2$, where $i = 1, 2, \cdots, k+1$. It is not difficult to see that if $|V(G_{i})|=2$ then
$G_{i}\cong Z_{3}$.
 Let $G''_{1}$, $\cdots$, $G''_{k}$ and  $G''_{k+1}$ be
the subgraphs derived from $G$ dy removing all 2-edge cuts in $G$, and which satisfy $V(G''_{1})=V(G_{1})$,$\cdots$, $V(G''_{k})=V(G_{k})$ and  $V(G''_{k+1})=V(G_{k+1})$.
For any graph $G$,
if we take every 2-edge cut in graph $G$ as an edge, and take the subgraphs $G''_{1}$, $\cdots$, $G''_{k+1}$ as a vertex, then graph $G$ can be viewed as a tree $T$. This implies that $T$ contains as least two leaves. In particular, each leaf of $T$ corresponds to some $G''_{i}$ with at least four vertices. Otherwise $G$ has  multiple edges, a contradiction.
Set that $k \geqslant \frac{n}{2} - 2$ 2-edge cuts in $G$. Then $n=|V(G)|\geqslant (\frac{n}{2} - 1 - 2) \times 2 + 2 \times 4 = n + 2 > n$,  a contradiction. So, $k \leqslant \frac{n}{2} - 3$, i.e., graph $G$ has at most $\frac{n}{2} - 3$ 2-edge cuts. Suppose that the graph $G$ contains exactly  $\frac{n}{2}-3$ 2-edge cuts.
If there exists as least three subgraphs   each of which  the number of vertices is greater than or equal to 4 in $\{G''_{1}$, $\cdots$, $G''_{k+1}\}$, then $n=|V(G)|\geqslant (\frac{n}{2} - 2 -3) \times 2 + 3 \times 4 = n + 2 > n$, a contradiction.  If  there exists two subgraphs in $\{G''_{1}$, $\cdots$, $G''_{k+1}\}$, and the number of vertices of one of two subgraphs  is greater than or equal to 6, then $n=|V(G)|\geqslant (\frac{n}{2} - 2 -2) \times 2 + 6+ 4 = n + 2 > n$, a contradiction. Combining arguments as above, we obtain that there just exists two subgraphs with 4 vertices, and the number of vertices of the other subgraphs are 2  in $G''_{1}$, $\cdots$, $G''_{k+1}$. This implies that $G\cong G^{*}$.
\end{proof}

\noindent{\bf Proof of Theorem \ref{the5}}.
By Lemma \ref{lem5},   we only consider 2-edge-connected simple cubic  graph.
Let $G$ be 2-edge-connected simple cubic  graph with only real root and $k$ 2-edge cuts, and let $F(G,t)\neq F(G^{*},t)$. From the proof of Lemma \ref{art4}, we know that
$G$ can be decomposed into $k + 1$ 3-edge-connected bridgeless cubic graphs $G_{1}, \cdots, G_{k+1}$,  and $|V(G_{i})|\geqslant 2$, where $i=1, 2,\cdots, k+1$. Notably, there are at least three graphs among them for which the number of vertices is greater than or equal to 4.  From the proof of Theorem \ref{th7}, we know that if $|V(G_{i})|=2$ (resp. $|V(G_{i})|=4$  or $|V(G_{i})|=6$) then $G_{i}\cong Z_{3}$ (resp. $G_{i}\cong K_{4}$  or $G_{i}\cong G_{2}$, where $G_{2}$ see Figure \ref{fig5}).
According to Lemmas \ref{lem12}, \ref{lem13} and \ref{lem16}, for $i=1, 2,\cdots, k+1$, the roots of flow polynomial of $F(G_{i},t)$ necessarily include  1 and 2, the remaining roots are greater than or equal to $\delta$. Thus, the set of flow roots of $F(G,t)$ is $S=\{t_{1}=1,\overbrace{t_{2}=2,\cdots,t_{k+2}=2}^{k+1},\overbrace{t_{k+3}\geqslant\delta,\cdots,t_{\frac{n}{2}+1}\geqslant\delta}^{\frac{n}{2} - k - 1}\}$, where $t_{i}$ denotes the $i$-th flow root and $\delta\approx  2.546$.
By Lemma \ref{lem12}, we obtain that the set of flow roots of $F(G^{*},t')$ is $S'=\{t'_{1}=1,\overbrace{t'_{2}=2,\cdots,t'_{\frac{n}{2}-1}=2}^{\frac{n}{2}-2},
t'_{\frac{n}{2}}=3, t'_{\frac{n}{2}+1}=3\}$.
For convenience, set
$
F(H,t)=t^{\frac{n}{2}+1}-\alpha_{1}t^{\frac{n}{2}}+\cdots+(-1)^{j}\alpha_{j}t^{\frac{n}{2}+1
-j}+\cdots+(-1)^{\frac{n}{2}+1}\alpha_{\frac{n}{2}+1}$ and $F(G^{*},t')=t'^{\frac{n}{2}+1}-\alpha'_{1}t'^{\frac{n}{2}}+\cdots+(-1)^{j}\alpha'_{j}t'^{\frac{n}{2}+1-j}
+\cdots+(-1)^{\frac{n}{2}+1}\alpha'_{\frac{n}{2}+1}
$.
Next, we will prove that $\alpha_{j}\geqslant \alpha'_{j}$, where $j=1, 2,\cdots, k+1$.

Assume that $S$   exactly contains 3 flow roots each of  which is greater than or equal to $\delta$, i.e., $S=\{t_{1}=1,t_{2}=2,\cdots,t_{\frac{n}{2} -2}=2,t_{\frac{n}{2} -1}\geqslant\delta,t_{\frac{n}{2}}\geqslant\delta,t_{\frac{n}{2}+1}\geqslant\delta\}$.
This implies that there are at most three graphs for which the number of vertices is greater than or equal to 4 in  $G_{1}, \cdots, G_{k+1}$. According to arguments as above, we know that there exists at least two graphs for which the number of vertices is greater than or equal to 4 in  $G_{1}, \cdots, G_{k+1}$. So, either there are three graphs in  $G_{1}, \cdots, G_{k+1}$, each of which has greater than or equal to  4 vertices, or there are two graphs  in   $G_{1}, \cdots, G_{k+1}$, with each having greater than or equal to  4 vertices. Suppose that there are three graphs each of which has greater than or equal to  4 vertices in  $G_{1}, \cdots, G_{k+1}$. By Lemmas \ref{lem13} and \ref{lem16}, and the  degree of the flow polynomial of a graph is equal to one-half times the number of vertices of the graph plus 1, we get that  these three graphs are all isomorphic to $K_{4}$. Similarly, assume that there are two graphs each of which has greater than or equal to  4 vertices in  $G_{1}, \cdots, G_{k+1}$.
By Lemmas \ref{lem13} and \ref{lem16}, and the  degree of the flow polynomial of a graph is equal to one-half times the number of vertices of the graph plus 1, we obtain that one of these two graphs is isomorphic to $K_{4}$, and the other is isomorphic to $G_{2}$, where $G_{2}$ is defined in Figure \ref{fig5}. Combining arguments as above, we know that $S=\{1,\overbrace{2,\cdots, 2}^{\frac{n}{2}-3},3,3,3\}$. By Lemma \ref{lem17}, we get that
\begin{eqnarray*}
\alpha_{j}=g_{j}(1,\overbrace{2,\cdots, 2}^{\frac{n}{2}-3},3,3,3)=\sum\limits_{1\leqslant i_{1}<i_{2}<\cdots<i_{j}\leqslant \frac{n}{2}+1}t_{i_{1}}t_{i_{2}}\cdots t_{i_{j}}\\
>\sum\limits_{1\leqslant i'_{1}<i'_{2}<\cdots<i'_{j}\leqslant \frac{n}{2}+1}t'_{i'_{1}}t'_{i'_{2}}\cdots t'_{i'_{j}}=g_{j}(1,\underbrace{2,\cdots,2}_{\frac{n}{2}-2},3,3)=\alpha'_{j}.
\end{eqnarray*}

Assume that $S$    contains at least 4 flow roots each of  which is greater than or equal to $\delta$, i.e., $S=\{t_{1}=1,\overbrace{t_{2}=2,\cdots,t_{\frac{n}{2}-k'+1}=2}^{\frac{n}{2} -k'},\overbrace{t_{\frac{n}{2}-k'+2}\geqslant\delta,\cdots, t_{\frac{n}{2}+1}\geqslant\delta}^{k'}\}$, where $k'\geqslant 4$.
Directly computing yields
\begin{eqnarray*}
\alpha_{j}&=&g_{j}(1,\overbrace{2,\cdots,2}^{\frac{n}{2} -k'},\overbrace{t_{\frac{n}{2}-k'+2}\geqslant\delta,\cdots, t_{\frac{n}{2}+1}\geqslant\delta}^{k'})=\sum\limits_{1\leqslant i_{1}<i_{2}<\cdots<i_{j}\leqslant \frac{n}{2}+1}t_{i_{1}}t_{i_{2}}\cdots t_{i_{j}}\\
&>&\sum\limits_{1\leqslant i''_{1}<i''_{2}<\cdots<i''_{j}\leqslant \frac{n}{2}+1}t''_{i''_{1}}t''_{i''_{2}}\cdots t''_{i''_{j}}=g_{j}(1,\overbrace{2,\cdots,2}^{\frac{n}{2} -4},2.5,2.5, 2.5,2.5)\\
&\geqslant&\sum\limits_{1\leqslant i'_{1}<i'_{2}<\cdots<i'_{j}\leqslant \frac{n}{2}+1}t'_{i'_{1}}t'_{i'_{2}}\cdots t'_{i'_{j}}=g_{j}(1,\overbrace{2,\cdots,2}^{\frac{n}{2} -2}, 3,3)=\alpha'_{j}.
\end{eqnarray*}
The proof of the theorem is now complete.
\hfill $\square$

\section{Conclusion}

In this paper, we establish an upper bound for the absolute value of the coefficient   $c_{i}$
corresponding to $t^{i}$ in the flow polynomial expansion of a graph. Furthermore, we present a lower bound for the absolute value of $c_{i}$ in the flow polynomial expansion of connected cubic graphs possessing exclusively real flow roots, and demonstrate the tightness of this bound through constructive examples. Through exhaustive investigation of all bridgeless connected cubic graphs with 10 or 12 vertices, we validate the conclusions presented in Theorems \ref{the4} and \ref{the5}.

We conclude by proposing two unresolved questions for further research:

\noindent \textbf{Question 1.}
For any connected bridgeless cubic graph $G$, does the inequality
 $$b_{i}\leqslant |c_{i}|$$
 hold between the absolute value of its flow polynomial coefficient
$c_{i}$  and the corresponding coefficient
$b_{i}$  in the expansion of $$(t+1)(t+2)^{\frac{n}{2}}?$$

\noindent \textbf{Question 2.}
For any simple connected bridgeless cubic graph  $G$, does the inequality
 $$b_{i}\leqslant |c_{i}|$$
 hold between the absolute value of its flow polynomial coefficient
$c_{i}$   and the corresponding coefficient
$b_{i}$  in the expansion of $$(t+1)(t+2)^{\frac{n}{2}-2}(t+3)^{2}?$$

\noindent{\bf Data Availability}\\
{No data were used to support this study.}
\\

\noindent{\bf Declaration of Competing Interest}

The authors declare that they have no known competing financial interests or personal relationships that could have appeared to influence the work reported in this paper.

\end{document}